\documentclass[10.5pt]{article}
\usepackage{amsfonts}
\usepackage{mathrsfs}
\usepackage{amssymb,amsmath,mathrsfs,amsthm}
\usepackage{amssymb}
\usepackage{graphicx}
\usepackage{indentfirst}
\usepackage{cite}
\usepackage[numbers,sort&compress]{natbib}
\usepackage{cases}
\usepackage{bm}
\usepackage{authblk}
\usepackage{setspace}
\usepackage{geometry}
\newtheorem{theorem}{Theorem}[section]
\newtheorem{lemma}[theorem]{Lemma}

\newtheorem{proposition}[theorem]{Proposition}
\theoremstyle{definition}
\newtheorem{definition}[theorem]{Definition}
\newtheorem{remark}[theorem]{Remark}

\theoremstyle{remark}

\numberwithin{equation}{section}

\begin{document}

\def\REFERENCE{\vspace*{.2in}
{\noindent\bf\heiti\zihao{6}References} \vspace*{.1in}}
\title{\Large \bf On Minkowskian Product of Finsler Manifolds}
\author{\normalsize Jiahui Li, Yong He\thanks{Corresponding author: heyong@xjnu.edu.cn}~~, Chang Tian, Na Zhang}
\affil{\small School of Mathematical Sciences, Xinjiang Normal University, Urumqi 830017, PR China}
\date{}
\vspace{-1em}
\maketitle
\par\noindent\textbf{Abstract:\enspace}Let~$(M_1,F_1)$ and~$(M_2,F_2)$ be a pair of Finsler manifolds. The Minkowskian product Finsler manifold~$(M,F)$ of~$(M_1,F_1)$ and~$(M_2,F_2)$ with respect to a product function~$f$ is the product manifold~$M=M_1\times M_2$ endowed with the Finsler metric~$F=\sqrt{f(K,H)}$, where~$K=F_1^2,H=F_2^2$. In this paper, the Cartan connection and Berwald connection of~$(M,F)$ are derived in terms of the corresponding objects of~$(M_1,F_1)$ and~$(M_2,F_2)$. Necessary and sufficient conditions for~$(M,F)$ to be Berwald (resp. weakly Berwald, Landsberg, weakly Landsberg) manifold are obtained. Thus an effective method for constructing special Finsler manifolds mentioned above is given.
\vspace{0.5em}
\vspace{0.5em}

\section{Introduction}

In a study on Finsler spaces admitting a group of motions of the greatest order, Ku\cite{GCH} has reached such a Finsler space with the linear element which is constituted by two Riemannian metrics of dimensions~1 and~$n-$1, respectively. Later on, this work was generalized to the case of two Riemannian metrics of multiple dimensions by Hu\cite{HHS}, and its geodesic equations were given by Su in \cite{SBQ}. Based on the above-mentioned studies, Okada\cite{OT} firstly proposed the definition of the Minkowskian product of Finsler manifolds and obtained its geodesics.

One of the important problems in Finsler geometry is to investigate manifolds with special curvature properties. Warped product and twisted product are effective ways used to construct special manifolds in Riemannian geometry and Finsler geometry. In~2016, He and Zhong\cite{HY1} studied doubly warped product complex Finsler manifolds, and constructed weakly complex Berwald manifold by doubly warped product. Peyghan, Tayebi and Nourmohammadi\cite{PT3} obtained the equivalent conditions for twisted product Finsler manifold to be Berwald (resp. weakly Berwald) manifold. Recently, Xiao, He, et al.\cite{XHLD,XHTL} gave the necessary and sufficient conditions for doubly twisted product complex Finsler manifolds to be complex Berwald (resp. weakly complex Berwald, complex Landsberg, complex Einstein-Finsler) manifold. One may wonder whether there are some other ways to construct some special Finsler manifolds such as Berwald manifold, weakly Berwald manifold, Landsberg manifold and weakly Landsberg manifold.

In \cite{WZC}, Wu and Zhong considered a class of product complex Finsler manifolds, which is defined as follows. Let~$(M_1,F_1)$ and~$(M_2,F_2)$ be a pair of complex Finsler manifolds and~$f(s,t)$ be a 1-homogeneous function on~$s$ and~$t$. Denote~$K=F_1^2,H=F_2^2$, one can define the fundamental function on product manifold~$M=M_1\times M_2$ by
\begin{equation}
F=\sqrt{f(K,H)}.
\end{equation}

Wu and Zhong\cite{WZC} investigated the possibility of the product complex Finsler manifold~$(M,F)$ with~$F$ given by~(1.1) to be complex Berwald manifold. Later, Xia and Wei\cite{XHC} systematically investigated the complex Finsler manifold endowed with the metric defined by~(1.1), and provided a possible way to construct complex Berwald (resp. weakly complex Berwald, complex Landsberg) manifold. Inspired by the above work, a natural problem is whether special Finsler manifolds can be constructed by Minkowskian product.

In this paper, our purpose is to study the necessary and sufficient conditions for the Minkowskian product Finsler manifold to be Berwald (resp. weakly Berwald, Landsberg, weakly Landsberg) manifold, thus we provide a new way to construct special Finsler manifolds mentioned above.

\section{Minkowskian product of Finsler manifolds}

Let~$M$ be a manifold of dimensions~$n$. We denote~$x=(x^{1},\cdots,x^{m})$ the local coordinates on~$M$, and~$(x,y)=(x^{1},\cdots,x^{n},y^{1},\cdots,y^{n})$ the local coordinates on the tangent bundle~$TM$ of~$M$. We shall assume that~$M$ is endowed with a Finsler metric~$F$ in the following sense.

\begin{definition}\hspace{-0.5em}\cite{AM}\enspace
A Finsler metric~$F$ on a manifold~$M$ is a function~$F$: $TM\rightarrow \mathbb{R}^+$ satisfying

(i)~$G=F^2$ is smooth on~$\tilde{M}=TM\setminus \{0\}$;

(ii)~$F(x,y)>0$ for any~$(x,y)\in \tilde{M}$;

(iii)~$F(x,\lambda y)=\vert \lambda \vert F(x,y)$ for any~$(x,y)\in TM$ and~$\lambda \in \mathbb{R}$;

(iv)~the Hessian matrix~$(G_{\alpha \beta}) =(\frac{\partial^{2}G}{\partial y^{\alpha}\partial y^{\beta}})$ is positive definite on~$\tilde{M}$.
\end{definition}

In this paper, we denote~$(G^{\alpha \beta})$ the inverse matrix of~$(G_{\beta \gamma})$ such that~$G^{\alpha \beta}G_{\beta \gamma}=\delta _{\gamma}^{\alpha}$. The derivatives of~$G$ with respect to the x-coordinates and y-coordinates are separated by semicolon; for instance,
\begin{equation}
G_{\alpha}=\frac{\partial G}{\partial y^\alpha},\quad G_{;\alpha}=\frac{\partial G}{\partial x^\alpha},\quad G_{\alpha;\beta}=\frac{\partial^{2}G}{\partial y^\alpha \partial x^\beta}.
\end{equation}

Let~$(M_1,F_1)$ and~$(M_2,F_2)$ be two Finsler manifolds with dimensions~$m$ and~$n$, respectively, then~$M=M_{1} \times M_{2}$ is product manifold with dimensions~$m+n$.

Let~$(x^{1},\cdots,x^{m})$ and~$(x^{m+1},\cdots,x^{m+n})$ be the local coordinates on~$M_{1}$ and~$M_{2}$, respectively,~$(x^{1},\\\cdots,x^{m},y^{1},\cdots,y^{m})$ and~$(x^{m+1},\cdots,x^{m+n},y^{m+1},\cdots,y^{m+n})$ be the induced local coordinates on the tangent bundles~$TM_1$ and~$TM_2$, respectively. Then the local coordinates on~$M$ are~$(x^{1},\cdots,x^{m+n})$, and the induced local coordinates on the tangent bundles~$TM$ are~$(x^{1},\cdots,x^{m+n},y^{1},\cdots,y^{m+n})$. Note that there is a natural isomorphism~$TM\cong TM_1\oplus TM_2$.

In the following, lowercase Greek indices such as~$\alpha,\beta,\gamma$, etc., will run from~1 to~$m+n$, lowercase Latin indices such as~$i,j,l$, etc., will run from~1 to~$m$, whereas lowercase Latin indices with a prime, such as~$i',j',l'$, etc., will run from~$m+1$ to~$m+n$, and the Einstein sum convention is assumed throughout this paper. The geometric objects associated to~$F_{1}$ and~$F_{2}$ are denoted with upper indices~1 or~2, for instance,~$\mathop{\check{\Gamma}_{j;l}^{i}}\limits^{1}$ and~$\mathop{\check{\Gamma}_{j';l'}^{i'}}\limits^{2}$ denote the Berwald connection coefficients associated to~$F_{1}$ and~$F_{2}$, respectively.

Let~$f:\left[0,+\infty\right)\times \left[0,+\infty\right)\rightarrow \left[0,+\infty\right)$ be a continuous function such as

(a)~$f(s,t)=0$ if and only if~$(s,t)=(0,0)$;

(b)~$f(\lambda s,\lambda t)=\lambda f(s,t)$ for any~$\lambda\in \left[0,+\infty\right)$;

(c)~$f$ is smooth on~$(0,+\infty)\times(0,+\infty)$;

(d)~$\frac{\partial f}{\partial s}\neq 0$,$\frac{\partial f}{\partial t}\neq 0$ for any~$(s,t)\in (0,+\infty)\times (0,+\infty)$;

(e)~$\frac{\partial f}{\partial s}\frac{\partial f}{\partial t}-2f\frac{\partial^{2} f}{\partial s\partial t}\neq 0$ for any~$(s,t)\in (0,+\infty)\times (0,+\infty)$.

\begin{definition}\hspace{-0.5em}\cite{OT}\enspace
Let~$(M_{1},F_{1})$ and~$(M_{2},F_{2})$ be two Finsler manifolds, and~$f$ be continuous function satisfying (a)-(e). Denote~$K=F_{1}^2$,~$H=F_{2}^2$, the Minkowskian product Finsler manifold~$(M,F)$ of~$(M_{1},F_{1})$ and~$(M_{2},F_{2})$ with respect to the product function~$f$ is the product manifold~$M=M_{1}\times M_{2}$ endowed with the Finsler metric~$F$: $TM\rightarrow \mathbb{R}^+$ given by
\begin{equation}
F\left(x,y\right)=\sqrt{f\left(K\left(x^{i},y^{i}\right),H\left(x^{i'},y^{i'}\right)\right)},
\end{equation}
where~$(x,y)\in TM,x=(x^{i},x^{i'})$,~$y=(y^{i},y^{i'})$ and~$(x^{i},y^{i})\in TM_1,(x^{i'},y^{i'})\in TM_2$. It is obvious that~$(M,F)$ is a Finsler manifold.
\end{definition}

Denote~$G=F^2$, then~(2.2) is equivalent to
\begin{equation}
G=F^2=f(K,H).
\end{equation}

By the 1-homogeneity of~$f(K,H)$, we have the following proposition.
\begin{proposition}\hspace{-0.5em}\cite{WZC}\enspace
\begin{align}
&f_{K}K+f_{H}H=f,  \\
f_{KK}K+f_{KH}H=0,\quad &f_{HK}K+f_{HH}H=0,\quad f^{2}_{KH}=f_{KK}f_{HH},
\end{align}
for~$K\neq0$ and~$H\neq0$. Where, we denote~$f_{K}=\frac{\partial f}{\partial K}$,~$f_{KH}=\frac{\partial^{2} f}{\partial K\partial H}$ and so on.
\end{proposition}

In the following, similar notations in~(2.1) are used for the functions~$K$ and~$H$, for instance,
$$K_i=\frac{\partial K}{\partial y^i},\quad K_{;i}=\frac{\partial K}{\partial x^i},\quad K_{i;j}=\frac{\partial^{2}K}{\partial y^i \partial x^j},\quad H_{i';j'}=\frac{\partial^{2}H}{\partial y^{i'} \partial x^{j'}}.$$

By Definition~2.2, one may easily obtain the following result.
\begin{proposition}
Let~$(M,F)$ be a Minkowskian product Finsler manifold of~$(M_1,F_1)$ and~$(M_2,F_2)$. Then the fundamental tensor matrix of~$G$ is given by
\begin{equation}
\bm{\mathbf{G}}=\left(G_{\alpha\beta}\right) =\left(\frac{\partial ^{2}G}{\partial y^{\alpha}\partial y^{\beta}}\right)=\left(
                                                                                                        \begin{array}{cc} G_{ij} & G_{ij'} \\
                                                                                                          G_{i'j} & G_{i'j'} \end{array}
                                                                                                      \right),
\end{equation}
where
\begin{equation}
\left\{
\begin{matrix} \begin{aligned} &\ G_{ij}=f_{K}K_{ij}+f_{KK}K_{i}K_{j}, \\
                      &\ G_{ij'}=f_{KH}K_{i}H_{j'}, \\
                      &\ G_{i'j}=f_{KH}H_{i'}K_{j}, \\
                      &\ G_{i'j'}=f_{H}H_{i'j'}+f_{HH}H_{i'}H_{j'}. \end{aligned}\end{matrix}
\right.
\end{equation}
\end{proposition}

For the inverse matrix of~$\bm{\mathbf{G}}$, we need the following lemmas.
\begin{lemma}\hspace{-0.5em}\cite{JCR}\enspace
Suppose a nonsingular matrix~$\bm{\mathbf{A}}\in M_{n}(\mathbb{R})$ has inverse~$\bm{\mathbf{A}}^{-1}$ and~$\bm{\mathbf{X}},\bm{\mathbf{Y}}\in\mathbb{R}^n$ are two column vector with~$\xi\in\mathbb{R}$ a constant. If~$\bm{\mathbf{B}}=\bm{\mathbf{A}}+\xi\bm{\mathbf{X}}\bm{\mathbf{Y}}^T$ is nonsingular, then $$\bm{\mathbf{B}}^{-1}=\bm{\mathbf{A}}^{-1}-\frac{\xi}{1+\xi\bm{\mathbf{Y}}^T \bm{\mathbf{A}}^{-1}\bm{\mathbf{X}}}\bm{\mathbf{A}}^{-1}\bm{\mathbf{X}} \bm{\mathbf{Y}}^T \bm{\mathbf{A}}^{-1},$$
where~$M_{n}(\mathbb{R})$ denotes the set of all~$n$-by-$n$ matrices over the real number field~$\mathbb{R}$.
\end{lemma}

\begin{lemma}\hspace{-0.5em}\cite{JCR}\enspace
Suppose~$\bm{\mathbf{Q}}\in M_{n}(\mathbb{R})$ is partitioned as~$\bm{\mathbf{Q}}=\left(\begin{array}{cc} \bm{\mathbf{A}} & \bm{\mathbf{B}} \\\bm{\mathbf{C}} & \bm{\mathbf{D}} \end{array}\right)$ with~$\bm{\mathbf{A}}\in M_{{n}_1}(\mathbb{R})$ and~$\bm{\mathbf{D}}\in M_{{n}_2}(\mathbb{R})$, where~$n_1+n_2=n$. Then the correspondingly partitioned presentation of~$\bm{\mathbf{Q}}^{-1}$ is
\begin{equation*}
\bm{\mathbf{Q}}^{-1}=\left(\begin{array}{cc}
  (\bm{\mathbf{A}}-\bm{\mathbf{B}} \bm{\mathbf{D}}^{-1} \bm{\mathbf{C}})^{-1} &
  \bm{\mathbf{A}}^{-1} \bm{\mathbf{B}} (\bm{\mathbf{C}} \bm{\mathbf{A}}^{-1} \bm{\mathbf{B}}-\bm{\mathbf{D}})^{-1} \\
  (\bm{\mathbf{C}} \bm{\mathbf{A}}^{-1} \bm{\mathbf{B}}-\bm{\mathbf{D}})^{-1} \bm{\mathbf{C}} \bm{\mathbf{A}}^{-1} &
  (\bm{\mathbf{D}}-\bm{\mathbf{C}} \bm{\mathbf{A}}^{-1} \bm{\mathbf{B}})^{-1}
  \end{array}\right),
\end{equation*}
assuming that all the relevant inverses exist.
\end{lemma}

\begin{proposition}
Let~$(M,F)$ be a Minkowskian product Finsler manifold of~$(M_1,F_1)$ and~$(M_2,F_2)$. Then the inverse matrix of~$\bm{\mathbf{G}}$ is given by
\begin{align}
\bm{\mathbf{G}}^{-1}=\left(\begin{array}{cc} G^{ji} & G^{ji'} \\ G^{j'i} & G^{j'i'} \end{array}\right),
\end{align}
where
\begin{equation}
\left\{
\begin{matrix} \begin{aligned} &\ G^{ji}=\frac{1}{f_K} (K^{ji}-\frac{f_H f_{KK}}{\Delta}y^j y^i), \\
                      &\ G^{ji'}=-\frac{1}{\Delta}f_{KH} y^j y^{i'}, \\
                      &\ G^{j'i}=-\frac{1}{\Delta}f_{KH} y^{j'} y^i, \\
                      &\ G^{i'j'}=\frac{1}{f_H} (H^{j'i'}-\frac{f_K f_{HH}}{\Delta}y^{j'} y^{i'}), \end{aligned}\end{matrix}
\right.
\end{equation}
and~$\Delta=f_K f_H-2ff_{KH}$.
\end{proposition}

\begin{proof}
In the following, we denote~$\bm{\mathbf{K}}=(K_{ij})$ and~$\bm{\mathbf{H}}=(H_{i'j'})$ the fundamental tensor matrices of~$K$ and~$H$, respectively, and denote their inverse by~$\bm{\mathbf{K}}^{-1}$ and~$\bm{\mathbf{H}}^{-1}$, respectively. Now setting
~$$\bm{\mathbf{U}}=(K_i)^{T}=(K_{1},K_{2},\cdots,K_{m})^T,~~\bm{\mathbf{W}}=(H_{i'})^{T}=(H_{m+1},H_{m+2},\cdots,H_{m+n})^T,$$
then~(2.7) can be rewritten as
\begin{equation}
\left\{
\begin{matrix} \begin{aligned} &\ G_{ij}=f_{K} \bm{\mathbf{K}}+f_{KK} \bm{\mathbf{U}} \bm{\mathbf{U}}^T, \\
                      &\ G_{ij'}=f_{KH} \bm{\mathbf{U}} \bm{\mathbf{W}}^T, \\
                      &\ G_{i'j}=f_{KH} \bm{\mathbf{W}} \bm{\mathbf{U}}^T, \\
                      &\ G_{i'j'}=f_{H} \bm{\mathbf{H}}+f_{HH} \bm{\mathbf{W}} \bm{\mathbf{W}}^T. \end{aligned}\end{matrix}
\right.
\end{equation}
For simplicity, we denote
\begin{align}
&\bm{\mathbf{A}}=f_{K} \bm{\mathbf{K}}+f_{KK} \bm{\mathbf{U}} \bm{\mathbf{U}}^T, \\
&\bm{\mathbf{B}}=f_{KH} \bm{\mathbf{U}} \bm{\mathbf{W}}^T,  \\
&\bm{\mathbf{C}}=f_{KH} \bm{\mathbf{W}} \bm{\mathbf{U}}^T, \\
&\bm{\mathbf{D}}=f_{H} \bm{\mathbf{H}}+f_{HH} \bm{\mathbf{W}} \bm{\mathbf{W}}^T,
\end{align}
then
\begin{equation}
\bm{\mathbf{G}}=\left(\begin{array}{cc} \bm{\mathbf{A}} & \bm{\mathbf{B}} \\
                                        \bm{\mathbf{C}} & \bm{\mathbf{D}} \end{array}\right).
\end{equation}
By Lemma~2.6, we can obtain
\begin{equation}
\bm{\mathbf{G}}^{-1}=
  \left(\begin{array}{cc}
  (\bm{\mathbf{A}}-\bm{\mathbf{B}} \bm{\mathbf{D}}^{-1} \bm{\mathbf{C}})^{-1} &
  \bm{\mathbf{A}}^{-1} \bm{\mathbf{B}} (\bm{\mathbf{C}} \bm{\mathbf{A}}^{-1} \bm{\mathbf{B}}-\bm{\mathbf{D}})^{-1} \\
  (\bm{\mathbf{C}} \bm{\mathbf{A}}^{-1} \bm{\mathbf{B}}-\bm{\mathbf{D}})^{-1} \bm{\mathbf{C}} \bm{\mathbf{A}}^{-1} &
  (\bm{\mathbf{D}}-\bm{\mathbf{C}} \bm{\mathbf{A}}^{-1} \bm{\mathbf{B}})^{-1}
  \end{array}\right).
\end{equation}
Since the 2-homogeneity of~$K$ and~$H$ with respect to the fiber coordinate~$y$, according to Euler's Theorem, we have
\begin{align}
&K_{ij}y^j=K_{i}, \\
&H_{i'j'}y^{j'}=H_{i'}, \\
&K_{i}y^i=2K, \\
&H_{i'}y^{i'}=2H.
\end{align}
Contracting (2.17) with~$K^{li}$ implies that
\begin{align}
y^l=K^{li}K_i.
\end{align}
Similarly, contracting (2.18) with~$H^{l'i'}$ yields
\begin{align}
y^{l'}=H^{l'i'}H_{i'}.
\end{align}
Using~(2.19) and~(2.21), we have
\begin{align}
\bm{\mathbf{U}}^T \bm{\mathbf{K}}^{-1} \bm{\mathbf{U}}=K_j K^{ji} K_i=K_{j}y^j=2K.
\end{align}
Similarly, using~(2.20) and~(2.22) implies
\begin{align}
\bm{\mathbf{W}}^T \bm{\mathbf{H}}^{-1} \bm{\mathbf{W}}=H_{j'} H^{j'i'} H_{i'}=H_{j'}y^{j'}=2H.
\end{align}
According to~(2.22) and Lemma 2.5, we can obtain
\begin{align}
\bm{\mathbf{A}}^{-1}&=\frac{1}{f_K}(\bm{\mathbf{K}}^{-1}-\frac{f_{KK}}{f_K+2K f_{KK}}\bm{\mathbf{K}}^{-1} \bm{\mathbf{U}} \bm{\mathbf{U}}^T \bm{\mathbf{K}}^{-1}),
\end{align}
which together with~(2.12), (2.13) and~(2.23) give
\begin{align}
\bm{\mathbf{C}} \bm{\mathbf{A}}^{-1} \bm{\mathbf{B}}=\left[\frac{2K f^2_{KH}}{f_K}-\frac{4K^2 f_{KK}f^2_{KH}}{f_K(f_K+2K f_{KK})} \right] \bm{\mathbf{W}} \bm{\mathbf{W}}^T.
\end{align}
Using~(2.14), (2.26) and the last equality of~(2.5), after a long but trivial computation, we have
\begin{align}
\bm{\mathbf{D}}-\bm{\mathbf{C}} \bm{\mathbf{A}}^{-1} \bm{\mathbf{B}}=f_H\bm{\mathbf{H}}+\frac{f_K f_{HH}}{f_K+2K f_{KK}} \bm{\mathbf{W}} \bm{\mathbf{W}}^T.
\end{align}
By~(2.24), (2.27) and Lemma 2.5, we obtain
\begin{align}
(\bm{\mathbf{D}}-\bm{\mathbf{C}} \bm{\mathbf{A}}^{-1} \bm{\mathbf{B}})^{-1}=
   \frac{1}{f_H}(\bm{\mathbf{H}}^{-1}-\frac{f_K f_{HH}}{f_H(f_K+2Kf_{KK})+2Hf_Kf_{HH}}\bm{\mathbf{H}}^{-1} \bm{\mathbf{W}} \bm{\mathbf{W}}^T \bm{\mathbf{H}}^{-1}).
\end{align}
Using~(2.4) and~(2.5), it is easy to check that
\begin{align}
f_H(f_K+2Kf_{KK})+2Hf_Kf_{HH}=f_K f_H-2ff_{KH}.
\end{align}
Plunging~(2.29) into~(2.28) and denote~$f_K f_H-2ff_{KH}=\Delta$, then
\begin{align}
(\bm{\mathbf{D}}-\bm{\mathbf{C}} \bm{\mathbf{A}}^{-1} \bm{\mathbf{B}})^{-1}
   &=\frac{1}{f_H}(\bm{\mathbf{H}}^{-1}-\frac{f_K f_{HH}}{\Delta}\bm{\mathbf{H}}^{-1} \bm{\mathbf{W}} \bm{\mathbf{W}}^T \bm{\mathbf{H}}^{-1}) \notag\\
   &=\frac{1}{f_H}(H^{j'i'}-\frac{f_K f_{HH}}{\Delta} H^{j'h'} H_{h'} H_{l'} H^{l'i'}) \notag\\
   &=\frac{1}{f_H}(H^{j'i'}-\frac{f_K f_{HH}}{\Delta} y^{j'} y^{i'}),
\end{align}
where in the last equality we used~(2.22).

Similar calculations give
\begin{align}
&(\bm{\mathbf{A}}-\bm{\mathbf{B}} \bm{\mathbf{D}}^{-1} \bm{\mathbf{C}})^{-1}
   =\frac{1}{f_K} (K^{ji}-\frac{f_H f_{KK}}{\Delta}y^j y^i), \\
&\bm{\mathbf{A}}^{-1} \bm{\mathbf{B}} (\bm{\mathbf{C}} \bm{\mathbf{A}}^{-1} \bm{\mathbf{B}}-\bm{\mathbf{D}})^{-1}
   =-\frac{1}{\Delta}f_{KH} y^j y^{i'}, \\
&(\bm{\mathbf{C}} \bm{\mathbf{A}}^{-1} \bm{\mathbf{B}}-\bm{\mathbf{D}})^{-1} \bm{\mathbf{C}} \bm{\mathbf{A}}^{-1}
   =-\frac{1}{\Delta}f_{KH} y^{j'} y^i.
\end{align}
It follows from~(2.8), (2.16) and~(2.30)-(2.33) that~(2.9).
\end{proof}

\section{Connections of Minkowskian product Finsler manifold}

Cartan connection and Berwald connection are two important connections in Finsler geometry\cite{ZCP}. In this section, we shall derive Cartan connection and Berwald connection of Minkowskian product Finsler manifold.

Let~$F$ be a Finsler metric, setting
\begin{align}
\mathbb{G}^\alpha=\frac{1}{2} G^{\alpha \beta}(G_{\beta;\gamma}y^\gamma-G_{;\beta}).
\end{align}
Then the Cartan nonlinear connection coefficients~$\Gamma^\alpha_\beta$ associated to~$F$ are given by\cite{AM}
\begin{align}
\Gamma^\alpha_\beta=\dot{\partial}_{\beta}(\mathbb{G}^\alpha).
\end{align}

In the following we denote
$$\partial_\alpha=\frac{\partial}{\partial x^\alpha}, \quad \dot{\partial}_\alpha=\frac{\partial}{\partial y^\alpha}.$$

Let~$\mathcal{V}$ be the vertical bundle of~$T(TM)$. The Cartan connection~$D:\mathcal{X}(\mathcal{V})\rightarrow\mathcal{X}(T^*\tilde{M}\otimes \mathcal{V})$ associated to a Finsler metric~$F$ was first introduced in~\cite{C}, and systemically studied in~\cite{AM}. The connection 1-forms~$\omega^\alpha_\beta$ of~$D$ are given by
$$\omega^\alpha_\beta=\Gamma^\alpha_{\beta;\gamma}dx^\gamma+\Gamma^\alpha_{\beta \gamma}\psi^\gamma,$$
where
\begin{align*}
&\Gamma^\alpha_{\beta;\gamma}=\frac{1}{2}G^{\alpha\mu}[\delta_\gamma(G_{\mu\beta})+\delta_\beta(G_{\mu\gamma})-\delta_\mu(G_{\beta\gamma})],\\
&\Gamma^\alpha_{\beta \gamma}=\frac{1}{2}G^{\alpha\mu}\dot{\partial}_\mu(G_{\beta\gamma}),
\end{align*}
and
$$\delta_\gamma=\partial_\gamma-\Gamma^\alpha_\gamma\dot{\partial}_\alpha, \quad \psi^\gamma=dy^\gamma+\Gamma^\gamma_\alpha dx^\alpha.$$


\begin{proposition}
Let~$(M,F)$ be a Minkowskian product Finsler manifold of~$(M_1,F_1)$ and~$(M_2,F_2)$. Then
\begin{align}
&G^{ih}K_h=\frac{1}{\Delta}(f_H-2Kf_{KH})y^i, \\
&G^{ih'}H_{h'}=-\frac{2}{\Delta}Hf_{KH}y^i, \\
&G^{ih'}H_{h'j'}=-\frac{2}{\Delta}f_{KH}H_{j'}y^i.
\end{align}
\end{proposition}

\begin{proof}
According to the first equality of~(2.9), and notice that~(2.4), (2.5) and (2.21), we have
\begin{align}
G^{ih}K_h&=\frac{1}{f_K}(K^{ih}-\frac{f_H f_{KK}}{\Delta}y^i y^h)K_h \notag\\
         &=f_K^{-1}(y^i-\frac{2}{\Delta}K f_H f_{KK}y^i) \notag\\
         &=\frac{1}{\Delta}(f_H-2ff_K^{-1}f_{KH}+2Hf_K^{-1}f_H f_{KH})y^i \notag\\
         &=\frac{1}{\Delta}(f_H-2Kf_{KH})y^i.\notag
\end{align}
Similarly, we obtain~(3.4) and~(3.5).
\end{proof}

\begin{proposition}
Let~$(M,F)$ be a Minkowskian product Finsler manifold of~$(M_1,F_1)$ and~$(M_2,F_2)$. Then
\begin{align}
\mathbb{G}^i=\mathop{\mathbb{G}^i} \limits^{1}, \quad \mathbb{G}^{i'}=\mathop{\mathbb{G}^{i'}} \limits^{2}.
\end{align}
\end{proposition}

\begin{proof}
By putting~$\alpha=i$ in~(3.1), and notice that~(2.3), we have
\begin{align}
\mathbb{G}^i&=\frac{1}{2} G^{i \beta}(G_{\beta;\gamma}y^\gamma-G_{;\beta}) \notag\\
            &=\frac{1}{2}\left[G^{ij}(G_{j;l}y^l+G_{j;l'}y^{l'}-G_{;j})+G^{ij'}(G_{j';l}y^l+G_{j';l'}y^{l'}-G_{;j'})\right] \notag\\
            &=\frac{1}{2}\left[G^{ij}(f_{K}K_{j;l}y^l+f_{KK}K_{;l}K_{j}y^l+f_{KH}H_{;l'}K_{j}y^{l'}-f_K K_{;j}) \right. \notag\\ 
            & \quad \left. +G^{ij'}(f_{H}H_{j';l'}y^{l'}+f_{HH}H_{;l'}H_{j'}y^{l'}+f_{HK}K_{;l}H_{j'}y^l-f_H H_{;j'})\right].
\end{align}
Plunging the first and second equalities of~(2.9) into~(3.7), and notice that~(2.4), (2.5) and~(3.3), we can obtain
\begin{align}
\mathbb{G}^i&=\frac{1}{2}\left[\frac{1}{f_K}(K^{ij}-\frac{f_H f_{KK}}{\Delta}y^i y^j)(f_{K}K_{j;l}y^l+f_{KK}K_{;l}K_{j}y^l+f_{KH}H_{;l'}K_{j}y^{l'}-f_K K_{;j}) \right. \notag\\ 
            & \quad \left. -\frac{1}{\Delta}f_{KH}y^i y^{j'}(f_{H}H_{j';l'}y^{l'}+f_{HH}H_{;l'}H_{j'}y^{l'}+f_{HK}K_{;l}H_{j'}y^l-f_H H_{;j'})\right] \notag\\
            &=\mathop{\mathbb{G}^i} \limits^{1}-\frac{1}{\Delta}\left[(Kf_{KK}f_{KH}+Hf^2_{KH})K_{;l}y^i y^l
                                                                      +(Hf_{KH}f_{HH}+Kf^2_{KH})H_{;l'}y^i y^{l'}\right] \notag\\
            &=\mathop{\mathbb{G}^i} \limits^{1}.\notag
\end{align}
Similar calculation gives~$\mathbb{G}^{i'}=\mathop{\mathbb{G}^{i'}} \limits^{2}$.
\end{proof}

\begin{proposition}
Let~$(M,F)$ be a Minkowskian product Finsler manifold of~$(M_1,F_1)$ and~$(M_2,F_2)$. Then the Cartan nonlinear connection coefficients associated to~$F$ are given by
\begin{equation*}
(\Gamma^\alpha_\beta)=\left(\begin{array}{cc} \Gamma^i_j & \Gamma^{i'}_{j} \\
                                              \Gamma^{i}_{j'} & \Gamma^{i'}_{j'} \end{array}\right),
\end{equation*}
where
\begin{align}
\Gamma^i_j=\mathop{\Gamma^i_j} \limits^{1}, \quad \Gamma^{i'}_{j'}=\mathop{\Gamma^{i'}_{j'}} \limits^{2}, \quad \Gamma^{i'}_{j}=\Gamma^{i}_{j'}=0.
\end{align}
\end{proposition}

\begin{proof}
By putting~$\alpha=i,\beta=j$ in~(3.2), and using Proposition~3.2, we have
\begin{align*}
\Gamma^i_j=\dot{\partial}_{j}(\mathbb{G}^i)=\dot{\partial}_{j}(\mathop{\mathbb{G}^i} \limits^{1})=\mathop{\Gamma^i_j} \limits^{1}.
\end{align*}
Similarly, we can obtain other equalities in~(3.8).
\end{proof}

The Berwald connection~$\check{D}:\mathcal{X}(\mathcal{V})\rightarrow\mathcal{X}(T^*\tilde{M}\otimes \mathcal{V})$ was first proposed by Berwald, and systemically studied in~\cite{Ab}. Its connection 1-forms can be expressed as
$$\check{\omega}^\alpha_\beta=\check{\Gamma}^\alpha_{\beta;\gamma}dx^\gamma,$$
where
\begin{align}
\check{\Gamma}^\alpha_{\beta;\gamma}=\dot{\partial}_\beta(\Gamma^\alpha_\gamma).
\end{align}

\begin{proposition}
Let~$(M,F)$ be a Minkowskian product Finsler manifold of~$(M_1,F_1)$ and~$(M_2,F_2)$,\\~$\check{\Gamma}^\alpha_{\beta;\gamma}$ are the coefficients of Berwald connection associated to~$F$. Then
\begin{align}
&\check{\Gamma}^i_{j;l}=\mathop{\check{\Gamma}^i_{j;l}} \limits^{1}, \quad \check{\Gamma}^{i'}_{j';l'}=\mathop{\check{\Gamma}^{i'}_{j';l'}} \limits^{2}, \\ &\check{\Gamma}^i_{j';l}=\check{\Gamma}^i_{j;l'}=\check{\Gamma}^i_{j';l'}=\check{\Gamma}^{i'}_{j';l}=\check{\Gamma}^{i'}_{j;l'}=\check{\Gamma}^{i'}_{j;l}=0.
\end{align}
\end{proposition}

\begin{proof}
By putting~$\alpha=i,\beta=j,\gamma=l$ in~(3.9), and using Proposition~3.3, we have
\begin{align*}
\check{\Gamma}^i_{j;l}=\dot{\partial}_j(\Gamma^i_l)=\dot{\partial}_j(\mathop{\Gamma^i_l} \limits^{1})=\mathop{\check{\Gamma}^i_{j;l}} \limits^{1}.
\end{align*}
Similar calculations give the rest of the equalities of Proposition~3.4.
\end{proof}

\section{Special Minkowskian product Finsler manifolds}

In this section, we shall give necessary and sufficient conditions for the Minkowskian product Finsler manifold to be Berwald (resp. weakly Berwald, Landsberg, weakly Landsberg) manifold. Thus we will give an new method to construct special Finsler manifolds mentioned above.

\begin{definition}\hspace{-0.5em}\cite{BCS2,MM}\enspace
A Finsler manifold~$(M,F)$ is called a Berwald manifold if locally the Berwald connection coefficients~$\check{\Gamma}^\alpha_{\beta;\gamma}(x,y)$ associated to~$F$ are independent of the fiber coordinate~$y$.
\end{definition}

\begin{theorem}
Let~$(M,F)$ be a Minkowskian product Finsler manifold of~$(M_1,F_1)$ and~$(M_2,F_2)$. Then~$(M,F)$ is a Berwald manifold if and only if~$(M_1,F_1)$ and~$(M_2,F_2)$ are both Berwald manifolds.
\end{theorem}

\begin{proof}
According to Definition~4.1,~$(M,F)$ is a Berwald manifold if and only if~$\check{\Gamma}^\alpha_{\beta;\gamma}=\check{\Gamma}^\alpha_{\beta;\gamma}(x)$. Thus by the relations of the Berwald connection in Proposition~3.4, we have
$$\mathop{\check{\Gamma}^i_{j;l}} \limits^{1}=\mathop{\check{\Gamma}^i_{j;l}} \limits^{1}(x),\quad \mathop{\check{\Gamma}^i_{j;l}} \limits^{2}=\mathop{\check{\Gamma}^i_{j;l}} \limits^{2}(x),$$
which is equivalent to the condition that~$(M_1,F_1)$ and~$(M_2,F_2)$ are both Berwald manifolds.
\end{proof}

\begin{remark}
Theorem~4.2 provides us an effective method to construct Berwald manifold.
\end{remark}

\begin{definition}\hspace{-0.5em}\cite{SZM1}\enspace
Let~$(M,F)$ be a~Finsler manifold,~$B^\alpha_{\beta \gamma \eta}$ and~$E_{\beta \gamma}$ are the coefficients of Berwald curvature and mean Berwald curvature of~$(M,F)$, respectively. Then~$(M,F)$ is called a weakly Berwald manifold if~$E_{\beta \gamma}\equiv 0$, where
\begin{align}
&E_{\beta \gamma}=\frac{1}{2} B^\alpha_{\beta \gamma \alpha}, \\
&B^\alpha_{\beta \gamma \eta}=\frac{\partial^3 \mathbb{G}^\alpha}{\partial y^{\beta}\partial y^{\gamma}\partial y^{\eta}}.
\end{align}
\end{definition}

Using~(4.2) and Proposition~3.2, by a straight forward computation, we obtain the following proposition.
\begin{proposition}
Let~$(M,F)$ be a Minkowskian product Finsler manifold of~$(M_1,F_1)$ and~$(M_2,F_2)$. Then the coefficients of Berwald curvature of~$(M,F)$ are given by
\begin{align}
&B^i_{jlh}=\mathop{B^i_{jlh}} \limits^{1},\quad B^{i'}_{j'l'h'}=\mathop{B^{i'}_{j'l'h'}} \limits^{2},\\
&B^i_{j'lh}=B^i_{jl'h}=B^i_{jlh'}=B^i_{j'l'h}=B^i_{j'lh'}=B^i_{jl'h'}=B^i_{j'l'h'}=0,\\
&B^{i'}_{j'lh}=B^{i'}_{jl'h}=B^{i'}_{jlh'}=B^{i'}_{j'l'h}=B^{i'}_{j'lh'}=B^{i'}_{jl'h'}=B^{i'}_{jlh}=0.
\end{align}
\end{proposition}

By Proposition~4.5, one may easily establish the following result.
\begin{proposition}
Let~$(M,F)$ be a Minkowskian product Finsler manifold of~$(M_1,F_1)$ and~$(M_2,F_2)$. Then the coefficients of mean Berwald curvature of~$(M,F)$ are given by
\begin{equation*}
(E_{\beta \gamma})=\left(\begin{array}{cc} E_{ij} & E_{i'j} \\
                                           E_{ij'} & E_{i'j'} \end{array}\right),
\end{equation*}
where
\begin{align*}
E_{ij}=\mathop{E_{ij}} \limits^{1}, \quad E_{i'j'}=\mathop{E_{i'j'}} \limits^{2}, \quad E_{i'j}=E_{ij'}=0.
\end{align*}
\end{proposition}

\begin{theorem}
Let~$(M,F)$ be a Minkowskian product Finsler manifold of~$(M_1,F_1)$ and~$(M_2,F_2)$. Then~$(M,F)$ is a weakly Berwald manifold if and only if~$(M_1,F_1)$ and~$(M_2,F_2)$ are both weakly Berwald manifolds.
\end{theorem}

\begin{proof}
According to Definition~4.4 and Proposition~4.6,~$(M,F)$ is a weakly Berwald manifold if and only if
$$\mathop{E_{ij}} \limits^{1}=0, \quad \mathop{E_{i'j'}} \limits^{2}=0,$$
which is equivalent to the condition that~$(M_1,F_1)$ and~$(M_2,F_2)$ are both weakly Berwald manifolds.
\end{proof}

\begin{remark}
Theorem~4.7 provides us an effective method to construct weakly Berwald manifold.
\end{remark}

\begin{definition}\hspace{-0.5em}\cite{SZM2,SYB}\enspace
Let~$(M,F)$ be a~Finsler manifold,~$L_{\alpha \beta \gamma}$ and~$J_{\alpha}$ are the coefficients of Landsberg curvature and mean Landsberg curvature of~$(M,F)$, respectively. Then~$(M,F)$ is called a Landsberg manifold if~$L_{\alpha \beta \gamma}\equiv 0$, where
\begin{align}
L_{\beta \gamma \eta}=-\frac{1}{4} y^\nu G_{\nu\alpha} B^\alpha_{\beta \gamma \eta}.
\end{align}
$(M,F)$  is called a weakly Landsberg manifold  if~$J_{\alpha}\equiv 0$, where
\begin{align}
J_{\alpha}=2G^{\beta \gamma} L_{\alpha \beta \gamma}.
\end{align}
\end{definition}

\begin{theorem}
Let~$(M,F)$ be a Minkowskian product Finsler manifold of~$(M_1,F_1)$ and~$(M_2,F_2)$. Then~$(M,F)$ is a Landsberg manifold if and only if~$(M_1,F_1)$ and~$(M_2,F_2)$ are both Landsberg manifolds.
\end{theorem}

\begin{proof}
By putting~$\beta=i,\gamma=j,\eta=l$ in~(4.6), and plunging~(4.3) and~(4.5) into it, we have
\begin{align}
L_{ijl}&=-\frac{1}{4} y^\nu G_{\nu\alpha} B^\alpha_{ijl} \notag\\
&=-\frac{1}{4}[(y^t G_{th}+y^{t'} G_{t'h})B^h_{ijl}+(y^t G_{th'}+y^{t'} G_{t'h'})B^{h'}_{ijl}] \notag\\
&=-\frac{1}{4}(y^t G_{th}+y^{t'} G_{t'h}) \mathop{B^h_{ijl}} \limits^{1} \notag\\
&=-\frac{1}{4}f_K K_h \mathop{B^h_{ijl}} \limits^{1},
\end{align}
where in the last equality we used
\begin{align}
y^t G_{th}+y^{t'} G_{t'h}&=y^t(f_{KK}K_tK_h+f_KK_{th})+y^{t'}(f_{KH}H_{t'}K_h) \notag\\
                         &=2Kf_{KK}K_h+f_KK_h-2Kf_{KK}K_h \notag\\
                         &=f_KK_h.\notag
\end{align}
It follows from~(4.6) and~(2.17) that~$\mathop{L_{ijl}} \limits^{1}=-\frac{1}{4}K_h \mathop{B^h_{ijl}}\limits^1$. Then~(4.8) reduces to
\begin{align}
L_{ijl}=f_K \mathop{L_{ijl}} \limits^{1}.
\end{align}
Similar calculations give us
\begin{align}
&L_{i'j'l'}=f_H \mathop{L_{i'j'l'}} \limits^{2}, \\
&L_{i'jl}=L_{ij'l}=L_{ijl'}=L_{i'j'l}=L_{i'jl'}=L_{ij'l'}=0.
\end{align}
According to~(4.9)-(4.11) and Definition~4.9,~$(M,F)$ is a Landsberg manifold if and only if
$$f_K \mathop{L_{ijl}} \limits^{1}=0, \quad f_H \mathop{L_{i'j'l'}} \limits^{2}=0.$$
Notice that~$f_K\neq 0,f_H\neq 0$ for~$K\neq 0,H\neq 0$, thus~$\mathop{L_{ijl}} \limits^{1}=0,\mathop{L_{i'j'l'}} \limits^{2}=0$, which is equivalent to the condition that~$(M_1,F_1)$ and~$(M_2,F_2)$ are both Landsberg manifolds.
\end{proof}

\begin{theorem}
Let~$(M,F)$ be a Minkowskian product Finsler manifold of~$(M_1,F_1)$ and~$(M_2,F_2)$. Then~$(M,F)$ is a weakly Landsberg manifold if and only if~$(M_1,F_1)$ and~$(M_2,F_2)$ are both weakly Landsberg manifolds.
\end{theorem}

\begin{proof}
By setting~$\alpha=i$ in~(4.7), we have
\begin{align}
J_i&=2G^{\beta \gamma} L_{i \beta \gamma} \notag\\
   &=2G^{jh} L_{ijh}+2G^{j'h} L_{ij'h}+2G^{jh'} L_{ijh'}+2G^{j'h'} L_{ij'h'}.
\end{align}
Plunging~(4.9), (4.11) and~(2.9) into~(4.12) and notice that~$\mathop{B^l_{ijh}} \limits^{1}y^j=0$, it follows that
\begin{align}
J_i&=2f^{-1}_K(K^{jh}-\frac{f_H f_{KK}}{\Delta}y^j y^h)f_K \mathop{L_{ijh}} \limits^{1} \notag\\
   &=\mathop{J_i} \limits^{1}+\frac{f_H f_{KK}}{2\Delta}K_l \mathop{B^l_{ijh}} \limits^{1} y^j y^h \notag\\
   &=\mathop{J_i} \limits^{1}.
\end{align}
Similarly, we get
\begin{align}
J_{i'}=\mathop{J_{i'}} \limits^{2}.
\end{align}
According to~(4.13), (4.14) and Definition~4.9,~$(M,F)$ is a weakly Landsberg manifold if and only if
\begin{align}
\mathop{J_i} \limits^{1}=0, \quad \mathop{J_{i'}} \limits^{2}=0,
\end{align}
which is equivalent to the condition that~$(M_1,F_1)$ and~$(M_2,F_2)$ are both weakly Landsberg manifolds.
\end{proof}

\begin{remark}
Theorem~4.10 and Theorem~4.11 provide us an effective way to construct Landsberg manifold and weakly Landsberg manifold, respectively.
\end{remark}

\noindent{\bf Acknowledgement}
This work is supported by National Natural Science Foundation of China (Grant Nos. 11761069).


\begin{thebibliography}{999999}
\bibitem{AM} M. Abate, G. Patrizio, Finsler Metrics$-$A Global Approach with Applications to Geometric Function Theory, Lecture Notes in Mathematics, vol. 1591, Springer-Verlag, Berlin/Heidelberg, 1994.
\bibitem{Ab} M. Abate, A characterization of the Chern and Berwald connections, Houston J. Math. 22 (4) (1996) 701-717.
\bibitem{C} E. Cartan, Les Espaces de Finsler, Hermann, Paris, 1934.
\bibitem{GCH} C. Ku, On Finsler spaces admitting a group of motions of the greatest order, Science Record (n.s.) (1) (1957) 215-218 (in Chinese).
\bibitem{HHS} H. Hu, Finsler product of two Riemannian space, Science Record (n.s.) (3) (1959) 446-448 (in Chinese).
\bibitem{JCR} R.A. Horn, C.R. Johnson, Matrix Analysis, Cambridge University Press, 1985.
\bibitem{HY1} Y. He, C. Zhong, On doubly warped product of complex Finsler manifolds, Acta Math. Sci. 36B (6) (2016) 1747-1766.
\bibitem{MM} M. Matsumoto, Foundations of Finsler Geometry and Special Finsler Spaces, Kaiseisha Press, Saikawa 3-23-2, Otsushi, Shigaken, Japan, 1986.
\bibitem{BCS2} M. Matsumoto, Remarks on Berwald and Landsberg spaces, in: D. Bao, S.S. Chern, Z. Shen (Eds.), Finsler Geometry: Joint Summer Research Conference on Finsler Geometry, July 16-20, 1995, Seattle, Washington, Amer. Math. Soc., Providence, Rhode Island, 1996, pp. 79-82.
\bibitem{OT} T. Okada, Minkowskian product of Finsler spaces and Berwald connection, J. Math. Kyoto Univ. 22 (2) (1982) 323-332.
\bibitem{PT3} E. Peyghan, A. Tayebi, L. Nourmohammadi Far, On twisted products Finsler manifolds, ISRN Geom. 2013 (2) (2013), http://doi.org/10.1155/2013/732432.
\bibitem{SBQ} B. Su, On Finslerian product space of two Riemannian metrics, Journal of Fudan University (Natural Science) (2) (1959) 1-11 (in Chinese).
\bibitem{SZM1} Z. Shen, Differential Geometry of Spray and Finsler Spaces, Kluwer Academic Publishers, Dordrecht, 2001.
\bibitem{SZM2} Z. Shen, On a class of Landsberg metrics in Finsler geometry, Can. J. Math. 61 (6) (2009) 1357-1374.
\bibitem{SYB} Y. Shen, Z. Shen, Introduction to Modern Finsler Geometry, Higher Education Press, Beijing, 2013 (in Chinese).
\bibitem{WZC} Z. Wu, C. Zhong, Some results on product complex Finsler manifolds, Acta Math. Sci. 31B (4) (2011) 1541-1552.
\bibitem{XHC} H. Xia, Q. Wei, On product complex Finsler manifolds, Turk. J. Math. 43 (1) (2019) 422-438.
\bibitem{XHLD} W. Xiao, Y. He, X. Lu,, et al., On doubly twisted product of complex Finsler manifolds, J. Math. Study 55 (2) (2022) 158-179.
\bibitem{XHTL} W. Xiao, Y. He, C. Tian, et al., Complex Einstein-Finsler doubly twisted product metrics, J. Math. Anal. Appl. 509 (2) (2022) 125981.
\bibitem {ZCP} C. Zhong, On real and complex Berwald connections associated to strongly convex weakly K\"{a}hler-Finsler metric, Differ. Geom. Appl. 29 (3) (2011) 338-408.
\end{thebibliography}
\end{document}